\documentclass[11pt, reqno, english]{amsart}  
\usepackage[utf8]{inputenc}
\usepackage[T1]{fontenc}
\usepackage{amsmath,amsthm}
\usepackage{amsfonts,amssymb}
\usepackage{url}
\usepackage{mathtools}  
\usepackage[colorlinks=true,urlcolor=blue,linkcolor=red,citecolor=magenta]{hyperref}
\usepackage{enumerate, paralist}
\usepackage{tikz-cd}
\usepackage[left=1in,right=1in,top=1in,bottom=1in]{geometry}

\usepackage{xstring,ifthen,tabularx}
\usepackage{tikz}
\newcommand*\rc[1]{\tikz[baseline=(char.base)]{ \StrLen{#1}[\mystr]\ifthenelse{ \mystr > 1}{\hspace{-.07cm}\node[shape=circle,color=red,text=black,draw,inner sep=0pt](char){#1};}{\hspace{-.06cm}\node[shape=circle,color=red,text=black,draw,inner sep=1pt](char){#1};}}}

\newcounter{commentcounter}

\theoremstyle{plain}
\newtheorem{theorem}{Theorem}[section]
\newtheorem{lemma}[theorem]{Lemma}
\newtheorem{corollary}[theorem]{Corollary}
\newtheorem{proposition}[theorem]{Proposition}

\theoremstyle{definition}

\newtheorem{remark}[theorem]{Remark}

\newcommand{\R}{\mathbb{R}}
\newcommand{\C}{\mathbb{C}}
\newcommand{\Z}{\mathbb{Z}}

\newcommand{\KG}{\mathrm{KG}}
\newcommand{\Emb}{\mathrm{Emb}}
\newcommand{\AEmb}{\mathrm{AEmb}}
\newcommand{\Imm}{\mathrm{Imm}}

\newcommand{\co}{\mathrm{coind}}
\newcommand{\cohomind}{\mathrm{swh}}

\begin{document}
% -------------------------------------------------------------------%

\title[Spaces of embeddings]{Spaces of embeddings:  Nonsingular bilinear maps, chirality, and their generalizations}

% -------------------------------------------------------------------%

% -------------------------------------------------------------------%

\author{Florian Frick}
\address[FF]{Dept.\ Math.\ Sciences, Carnegie Mellon University, Pittsburgh, PA 15213, USA}
\email{frick@cmu.edu} 

\author{Michael Harrison}
\address[MH]{Dept.\ Math.\ Sciences, Carnegie Mellon University, Pittsburgh, PA 15213, USA}
\email{mah5044@gmail.com} 

\thanks{FF was supported by NSF grant DMS 1855591 and a Sloan Research Fellowship.}

% -------------------------------------------------------------------%

\begin{abstract}
\small
Given a space $X$ we study the topology of the space of embeddings of $X$ into $\R^d$ through the combinatorics of triangulations of~$X$. We give a simple combinatorial formula for upper bounds for the largest dimension of a sphere that antipodally maps into the space of embeddings. This result summarizes and extends results about the nonembeddability of complexes into~$\R^d$, the nonexistence of nonsingular bilinear maps, and the study of embeddings into $\R^d$ up to isotopy, such as the chirality of spatial graphs.
\end{abstract}

\date{October 22, 2020}
\maketitle

\section{Introduction}
\label{sec:intro}

Given a topological space~$X$, a classical question in geometric topology asks for the lowest dimension $d$ such that $X$ embeds into~$\R^d$. Among the milestones of this line of inquiry are the characterization of planar graphs (as graphs without $K_5$- and $K_{3,3}$-minors), Whitney's embedding theorem (that any smooth $d$-manifold embeds into~$\R^{2d}$), the Van Kampen--Flores theorem (that the $d$-skeleton of the $(2d+2)$-simplex does not embed into~$\R^{2d}$), and upper and lower bounds for the embedding dimension of real projective space~$\R P^n$. More generally, one may study the topology of the space $\Emb(X, \R^d)$ of all embeddings $X \hookrightarrow \R^d$ with the compact-open topology.  These spaces are particularly well-studied when $X$ is a discrete space on $n$ points, where $\Emb(X,\R^d)$ is the ordered configuration space of $n$ points in~$\R^d$.  For general $X$, we study the following question: if there is an embedding $X \hookrightarrow \R^d$, is there a rich collection of such embeddings, or will any sufficiently rich parametrized family of continuous maps necessarily include a nonembedding? 

The coindex $\co(Y)$ of a $\Z/2$-space $Y$ is the largest dimension $k$ such that a $\Z/2$-map ${S^k \to Y}$ from the $k$-dimensional sphere~$S^k$ equipped with the antipodal $\Z/2$-action exists.  The coindex may be considered as some measure of the size or complexity of a space, though it is only meaningful when the $\Z/2$-action is free, since otherwise the constant map from any given sphere to a fixed point is $\Z/2$-equivariant.

Given $\ell \in \{1,2,\dots,d\}$ we fix a $\Z/2$-action on $\R^d$ that flips the sign of the first $\ell$ coordinates, that is, the generator $\varepsilon$ acts by $\varepsilon \cdot (x_1, \dots, x_d) = (-x_1, \dots, -x_\ell, x_{\ell+1}, \dots, x_d)$. The space $\Emb(X,\R^d)$ with the naturally induced $\Z/2$-action $(\varepsilon \cdot f)(x) = \varepsilon \cdot f(x)$ will be denoted by $\Emb_\ell(X,\R^d)$.  Observe that this action is free if and only if $X$ does not embed into~$\R^{d-\ell}$.  

We formalize the question above by developing bounds for the coindex of $\Emb_\ell(X, \R^d)$.  Our interest in this specific parameter stems from the fact that it has been studied in various guises, albeit in an implicit manner.  Besides the existence question for embeddings, the two main examples of this previous study are \emph{chirality}, related to the case $\ell = 1$, and the existence of \emph{nonsingular bilinear maps}, related to the case $\ell = d$.  Our main result is a simple combinatorial statement which summarizes and extends previous results on nonembeddability, chirality, and nonsingular bilinear maps.  Before stating the main result, we discuss several of these applications.

\subsection{Chirality}
The lower bound $\co(\Emb_1(X,\R^d)) \ge 1$ simply means that some embedding $f\colon X \hookrightarrow \R^d$ can be isotoped into its \emph{mirror image}, an embedding obtained from $f$ by composing with an orientation-reversing homeomorphism.
Simon~\cite{Simon} exhibited a specific embedding of~$K_{3,3}$, the complete bipartite graph on $3+3$ vertices, into $\R^3$ that is \emph{chiral}, that is, it cannot be isotoped into its mirror image.  Flapan~\cite{Flapan1989} showed that every embedding $K_{3,3} \hookrightarrow \R^3$ is chiral. (Actually, Flapan showed that $K_{3,3}$ is chiral in an a priori stronger sense; we will explain the details in Section~\ref{subsec:chirality}.) 
This strengthens the nonplanarity of~$K_{3,3}$: If there was an embedding of $K_{3,3}$ into~$\R^2$, then thinking of $\R^2 \subset \R^3$ this embedding would be equal to its mirror image. Other classical nonembeddability results have been generalized in the same way. For example, it is known that no embedding $\R P^2 \hookrightarrow \R^4$ is isotopic to its mirror image (see~\cite{Auckly2018}) and the same is true for embeddings $\C P^2 \hookrightarrow \R^7$. See Skopenkov~\cite{Skopenkov2010} for a classification of smooth embeddings $\C P^2 \hookrightarrow \R^7$ up to smooth isotopy.

 As $\co(\Emb_1(X,\R^d)) \ge 1$ is equivalent to some embedding $X \hookrightarrow \R^d$ being isotopic to its mirror image, improved lower bounds for $\co(\Emb_\ell(X,\R^d))$ can be regarded as \emph{higher achirality}.  For odd $d$, $\co(\Emb_d(X,\R^d)) = 0$ is also equivalent to every embedding $f\colon X \hookrightarrow \R^d$ being chiral, since $-f$ is also a mirror image of~$f$. In this sense the following result extends the nonembeddability results for $\R P^2$ and $\C P^2$ and the chirality result for~$\C P^2$.
 
 \begin{corollary}
 \label{cor:coindex-rp2}
 	For any dimension $d\ge 1$, $$\co(\Emb_d(\R P^2, \R^d)) = 4\Big\lfloor \frac{d}{4} \Big\rfloor -1 \quad \text{and} \quad \co(\Emb_d(\C P^2, \R^d)) = \begin{cases} 8k \qquad \qquad d = 8k+7 \\ 8\Big\lfloor \frac{d}{8} \Big\rfloor - 1 \quad \text{otherwise} \end{cases}. $$
 \end{corollary}

For a simplicial complex $\Sigma$ a continuous map $f\colon \Sigma \to \R^d$ is an \emph{almost-embedding} if $f(\sigma) \cap f(\tau) = \emptyset$ for any two vertex-disjoint faces $\sigma$ and $\tau$ of~$\Sigma$.  As for embeddings, we write $\AEmb_\ell(\Sigma, \R^d)$ for the space of almost-embeddings $\AEmb(\Sigma,\R^d)$ equipped with the compact-open topology and the $\Z/2$-action that reverses the sign of the first $\ell$ coordinates.  In fact, Corollary~\ref{cor:coindex-rp2} holds for the space of almost-embeddings of the minimal triangulations of projective planes over $\R$ and~$\C$. We also generalize the chirality of embeddings $\R P^2 \hookrightarrow \R^4$ and $\C P^2 \hookrightarrow \R^7$ by extending results to almost-embeddings.  

\begin{corollary}
\label{cor:rp2}
	Let $\Sigma_{\R P^2}$ and $\Sigma_{\C P^2}$ denote the minimal triangulations of $\R P^2$ and $\C P^2$, respectively. Given any almost-embedding $\Sigma_{\R P^2} \to \R^4$ there is no homotopy to its mirror image through almost-embeddings. The same is true for any almost-embedding $\Sigma_{\C P^2} \to \R^7$.  That is,
\[
\co(\AEmb_1(\Sigma_{\R P^2}, \R^4)) = 0 \quad \text{and} \quad \co(\AEmb_1(\Sigma_{\C P^2}, \R^7)) = 0.
\]
\end{corollary}

The Van Kampen--Flores theorem~\cite{Flores1933, VanKampen1933} generalizes the nonplanarity of $K_{3,3}$ and of the complete graph $K_5$ to higher-dimensional simplicial complexes: It asserts that the $k$-skeleton $\Delta_{2k+2}^{(k)}$ of the $(2k+2)$-simplex $\Delta_{2k+2}$ and the $(k+1)$-fold join $[3]^{*(k+1)}$ of a $3$-point space do not embed into~$\R^{2k}$. The case $k=1$ is the nonplanarity of $K_5$ and of~$K_{3,3}$. We deduce a chiral generalization of this classical nonembeddability result:

\begin{corollary}
\label{cor:vKF}
	Given any almost-embedding $\Delta_{2k+2}^{(k)} \to \R^{2k+1}$ there is no homotopy to its mirror image through almost-embeddings. The same is true for any almost-embedding $[3]^{*(k+1)} \to \R^{2k+1}$.  That is,
	\[
\co(\AEmb_1(\Delta_{2k+2}^{(k)}, \R^{2k+1})) = 0 \quad \text{and} \quad \co(\AEmb_1([3]^{*(k+1)}, \R^{2k+1})) = 0.
\]
\end{corollary}

\subsection{Nonsingular bilinear maps}
A bilinear map $B\colon \R^{p+1} \times \R^{q+1} \to \R^d$ is called \emph{nonsingular} if ${B(x,y) = 0}$ implies that $x = 0$ or $y = 0$.  Nonsingular bilinear maps are related to tangent vector fields on spheres, immersions and embeddings of real projective spaces, skew fibrations, and totally nonparallel immersions (see Section \ref{sec:nonsingularhistory} for a brief history), and hence have been studied extensively.  Nevertheless, the triples $(p+1,q+1,d)$ of dimensions which admit nonsingular bilinear maps are far from classified.

Given a nonsingular bilinear map $B\colon \R^{p+1} \times \R^{q+1} \to \R^d$, the restriction to the unit ball $D^{p+1}$ in $\R^{p+1}$ and the unit sphere $S^q$ in $\R^{q+1}$ yields a map $D^{p+1} \times S^q \to \R^d$ that is $\Z/2$-equivariant, or \emph{skew}, in the second entry. Moreover, for any fixed $y_0 \in S^q$ the restriction $D^{p+1} \times \{y_0\} \to \R^d$ is an embedding, and thus by currying the map $D^{p+1} \times S^q \to \R^d$ we see that the existence of a nonsingular bilinear map $\R^{p+1} \times \R^{q+1} \to \R^d$ implies $\co(\Emb_d(D^{p+1}, \R^d)) \ge q$.  Similarly, a nonsingular bilinear map induces a $(\Z/2)^2$-equivariant (or \emph{biskew}) map $S^p \times S^q \to \R^d$ which avoids zero, hence $\co(\Emb_d(S^p, \R^d)) \ge q$.

Using known results for bilinear and biskew maps, it is straightforward to show that, if the integers $p$ and $d-p$ do not share any common ones in their binary expansions, then
\begin{align}
\label{eqn:embsm}
\co(\Emb_d(S^p, \R^d)) = d - p -1;
\end{align}
see Corollary \ref{cor:embsm}.

We will prove a discretized generalization of this result:

\begin{corollary}
\label{cor:parametrized-radon}
	If the nonnegative integers $p$ and $d-p$ do not share any common ones in their binary expansions then $\co(\AEmb_d(\partial \Delta_{p+1}, \R^{d})) = d-p-1$.
\end{corollary}	

It follows from the Borsuk-Ulam theorem that $\Emb(S^n, \R^n)$ is empty; this is the special case $p=d$ of (\ref{eqn:embsm}).  The topological Radon theorem of Bajm\'oczy and B\'ar\'any~\cite{BB}  
asserts that there is no almost-embedding $ \partial\Delta_{n+1} \to \R^n$; this is the special case $p = d$ of Corollary \ref{cor:parametrized-radon}.

More generally, the result that any biskew map $S^2 \times S^1 \to \R^3$ must have a zero implies that when turning $S^2$ inside-out, which is possible through immersions by work of Smale \cite{Smale}, then at some point in time two antipodal points of $S^2$ are mapped to the same point. Corollary~\ref{cor:parametrized-radon} in this case asserts that when turning the boundary of a tetrahedron inside-out, two disjoint faces of the tetrahedron must overlap in the image at some point in time.

Together with existence results for nonsingular bilinear maps, Corollary \ref{cor:parametrized-radon} yields the exact values of $\co(\AEmb_d(\partial\Delta_{p+1}, \R^{d}))$ for $p \leq 8$, generalizing classical nonexistence results for nonsingular bilinear maps, as tabulated by Berger and Friedland~\cite{BergerFriedland}; see Figure~\ref{fig:embsm} (in Section~\ref{sec:nonsingular}).

\subsection{Statement of the main result}
Our main result summarizes and extends results on chirality of embeddings and the nonexistence of nonsingular bilinear maps in one simple combinatorial statement. For a simplicial complex $\Sigma$ on ground set $[n] = \{1,2,\dots,n\}$, we denote by $\KG(\Sigma)$ the \emph{Kneser graph of its nonfaces}, that is, the graph whose vertices correspond to those subsets of $[n]$ that do not form a face of~$\Sigma$, and with edges between vertices corresponding to disjoint faces. For a graph $G$ we denote by $\chi(G)$ its \emph{chromatic number}, that is, the least number of colors $c$ needed to color its vertices such that the two endpoints of every edge receive distinct colors.  A proper coloring of the Kneser graph of minimal nonfaces induces a proper coloring of the Kneser graph of all nonfaces (by coloring some nonface $\sigma$ by the color of some minimal nonface $\tau \subset \sigma$).

Our main result on the coindex of the space of embeddings is the following:

\begin{theorem}
\label{thm:coindex}
	Let $\Sigma$ be a simplicial complex on ground set~$[n]$. Let $d$ and $\ell$ be nonnegative integers such that $m = d-n+\chi(\KG(\Sigma))+2 \le \ell$. If the integers $m$ and $\ell-m$ do not share common ones in their binary expansions, then $\AEmb_\ell(\Sigma, \R^d)$, and thus $\Emb_\ell(\Sigma, \R^d)$, has coindex at most~${m-1}$.
\end{theorem}

When the parameter $\ell$ is removed from the statement, the case $m=0$ is a far-reaching nonembeddability result known as Sarkaria's Coloring/Embedding Theorem, which can be found in Matou\v sek's book~\cite[Theorem 5.8.2]{Matousek} and is implicit in Sarkaria's papers \cite{Sarkaria90, Sarkaria91a}.  Observe that any nonembeddability result which follows from Sarkaria's Theorem also admits a chiral generalization by taking $\ell = 1$ in Theorem \ref{thm:coindex}.  We discuss this and other consequences, including the corollaries advertised above, in Section~\ref{sec:appl}.  We prove Theorem~\ref{thm:coindex} in Section~\ref{sec:proof}.

\section{Some context}

To contextualize Theorem~\ref{thm:coindex} and its corollaries, we provide brief introductions to chirality and nonsingular bilinear maps.

\subsection{Chirality} 
\label{subsec:chirality}

The study of whether a given embedding $G \hookrightarrow \R^3$ of a graph~$G$ is chiral was originally motivated by chemistry. In the original motivation $G$ represents the bonds in a molecule. See~\cite{Flapan2000} for an introduction, where for instance we can learn that ``one enantiomer of the molecule called carvone smells like caraway, whereas its mirror image smells like spearmint.'' 

In order to define chirality we need the notions of isotopy and ambient isotopy, which we recall now. A continuous map $F\colon X \times [0,1] \to \R^d, (x,t) \mapsto f_t(x)$ such that $f_t$ is an embedding for every $t$ is an \emph{isotopy} between $f_0$ and~$f_1$. Given two embedding $h, h' \colon X\hookrightarrow \R^d$, an \emph{ambient isotopy} from $h$ to $h'$ is a continuous map $F\colon \R^d \times [0,1] \to \R^d, (x,t) \mapsto f_t(x)$ such that $f_t$ is a homeomorphism for every~$t$, $f_0$ is the identity and~${f_1(h(x))= h'(x)}$ for all $x\in X$. In the smooth category any isotopy can be extended to an ambient isotopy. This fails without the smoothness requirement; for example, any knot $S^1 \hookrightarrow \R^3$ is isotopic to the unknot. 

The notion of ambient isotopy can similarly be defined for subspaces instead of embeddings: Given two subspaces $X, X' \subset \R^d$, an \emph{ambient isotopy} from $X$ to $X'$ is a continuous map ${F\colon \R^d \times [0,1] \to \R^d}$, $(x,t) \mapsto f_t(x)$ such that $f_t$ is a homeomorphism for every~$t$, $f_0$ is the identity and~${f_1(X)= X'}$. An ambient isotopy between the embeddings $h$ and $h'$ is also an ambient isotopy from $h(X)$ to~$h'(X)$, so an ambient isotopy between subspaces (the images of two given embeddings) is in general a weaker requirement than an ambient isotopy between embeddings.

Let $X \subset \R^d$, and let $h \colon \R^d \to \R^d$ be a homeomorphism which is isotopic to a reflection in a hyperplane, that is, $h$ is an orientation-reversing homeomorphism. Then $h(X)$ is a \emph{mirror image} of~$X$. All mirror images are the same up to ambient isotopy. The spatial graphs literature calls a graph $G$ (intrinsically) \emph{chiral} if given any embedding $f \colon G \hookrightarrow \R^3$ there is no ambient isotopy of the subspace $f(G) \subset \R^3$ to a mirror image of~$f(G)$. We will refer to this as \emph{unlabelled chirality} to distinguish this from the notion of chirality for embeddings.

It is not difficult to see that all embeddings $X \hookrightarrow \R^d$ of some space $X$ are chiral if and only if~${\co(\Emb_1(X,\R^d)) \le 0}$. More generally, we have the following.

\begin{lemma}
	Let $X$ be a space, and let $0 \le \ell \le d$ be integers. Then $\co(\Emb_\ell(X, \R^d)) \ge 1$ if and only if there is an isotopy between some embedding $f\colon X \hookrightarrow \R^d$ and the embedding obtained from $f$ by flipping the sign of the first $\ell$ coordinates.
\end{lemma}

\begin{proof}
	Let $\Phi\colon S^1 \to \Emb_\ell(X, \R^d)$ be a $\Z/2$-map. Denote the generator of the $\Z/2$-action on $\Emb_\ell(X, \R^d)$ by~$\varepsilon$. Fix a point $x_0 \in S^1$ and a path $\gamma \colon[0,1]\to S^1$ with $\gamma(0) = x_0$ and $\gamma(1) = -x_0$. Then $\Phi\circ \gamma$ is an isotopy from the embedding $\Phi(x_0)$ to~$\Phi(-x_0) = \varepsilon \Phi(x_0)$. Conversely, let $F\colon [0,1] \to \Emb(X, \R^d)$ be an isotopy with $F(1) = \varepsilon F(0)$. Think of $[0,1]$ as the upper semi-circle of $S^1$ and extend $F$ to a continuous map $\Phi\colon S^1 \to \Emb(X,\R^d)$ by setting $\Phi(x) = F(x)$ for $x$ in the upper semi-circle and $\Phi(x) = \varepsilon \Phi(-x)$ for $x$ in the lower semi-circle.
\end{proof}

Given a graph~$G$, the notion of unlabelled chirality for $G$ is often deduced from the (labeled) version for embeddings by showing that any homeomorphism $G \to G$ (up to isotopy) extends to an orientation-preserving homeomorphism $\R^3 \to \R^3$. For example, this is how Simon~\cite{Simon} shows that the image of a particular embedding of $K_{3,3}$ is chiral. Flapan~\cite{Flapan1989} showed that $K_{3,3}$ is chiral in the stronger, unlabelled sense.

The notion of unlabelled chirality of spatial graphs $G \subset \R^3$ is different from the stronger notion of chirality of embeddings $G \hookrightarrow \R^3$. Flapan and Weaver~\cite{Flapan1992} showed that the complete graphs $K_{4n+3}$, $n \ge 1$, are chiral as subspaces, and all other complete graphs are not chiral in the unlabelled sense, that is, for every $n \not\equiv 3 \mod 4$ there is a subspace $X$ of $\R^3$ homeomorphic to $K_n$ and an ambient isotopy to a mirror image of~$X$. On the other hand, Theorem~\ref{thm:coindex} easily implies that for any given embedding $f\colon K_n \hookrightarrow \R^3$, $n\ge 5$, there is no isotopy to a mirror image of $f$ through almost-embeddings.

In this sense our results are weaker than the chirality results for spatial graphs since they apply to (labelled) embeddings and not to (unlabelled) subspaces. There is no difference for some spaces, such as for $\R P^2$ since any homeomorphism of $\R P^2$ is the identity up to isotopy. In fact, in those cases our results may be stronger since they assert the nonexistence of an isotopy, which is stronger than the nonexistence of an ambient isotopy as we do not have any smoothness requirement.

\subsection{Nonsingular bilinear maps}
\label{sec:nonsingularhistory}
A bilinear map $B\colon \R^{p+1} \times \R^{q+1} \to \R^d$ is  \emph{nonsingular} if $B(x,y) = 0$ implies $x = 0$ or $y=0$.  The \emph{Hurwitz--Radon} function $\rho(p+1,d)$ is defined as the maximum number $q+1$ such that there exists a nonsingular bilinear map $\R^{p+1} \times \R^{q+1} \to \R^d$.  The values of $\rho(p+1) \coloneqq \rho(p+1,p+1)$ are known and may be computed as follows: decompose $p+1 =2^{b+4c}\cdot(2a+1)$, with $0 \leq b < 4$, then $\rho(p+1)$ is equal to $2^b + 8c$.  For example, $\rho(4) = 4$, and a nonsingular bilinear map $\R^4 \times \R^4 \to \R^4$ is given by quaternionic multiplication.

The Hurwitz--Radon function originally appeared in the independent works of Hurwitz~\cite{Hurwitz} and Radon~\cite{Radon} in their studies of square identities.  It has since made prominent appearances in topology.  In particular, an important result of Adams \cite{Adams} is that the inequality $q+1 \leq \rho(p+1) - 1$ holds if and only if there exist $q+1$ linearly independent tangent vector fields on~$S^{p}$.  The Hurwitz--Radon function is also related to the study of \emph{skew fibrations} of $\R^n$ by pairwise skew affine copies of~$\R^{p+1}$; see \cite{Harrison, Harrison2, Harrison3, OvsienkoTabachnikov, OvsienkoTabachnikov2}.

Nonsingular (symmetric) bilinear maps have been studied, in part, due to their relationships to immersions and embeddings of projective spaces and to totally nonparallel immersions; see~\cite{Harrison5, James4}.  Most notably, they were studied in a series of articles by K.Y.\ Lam (e.g.\  \cite{Lam1, Lam6, Lam4, Lam3, Lam2, Lam5}) and by Berger and Friedland~\cite{BergerFriedland}.

We also note that nonsingular skew-linear maps $B \colon \R^{p+1} \times S^q \to \R^d$ (here \emph{skew} refers to $\Z/2$-equivariance in the second slot) and biskew maps $S^p \times S^q \to \R^d$ have received separate attention; the former notion is equivalent to the \emph{generalized vector field problem}, and the latter appears in the study of symmetric topological complexity \cite{Gonzalez}.  Each of these notions will appear in our study of coindex.

\section{Preliminaries and related results}
\label{sec:prelim}

Here we summarize auxiliary definitions and results that we will need in subsequent sections.

\subsection{Nonsingular bilinear and biskew maps}
\label{sec:nonsingularprelim}

We collect some simple existence results for nonsingular bilinear maps.

\begin{lemma}
\label{lem:bilinear-exist}
Nonsingular bilinear maps exist in the following dimensions:
\begin{compactenum}[(a)]
\item $\R \times \R^k \to \R^k$,
\item $\R^2 \times \R^{2k} \to \R^{2k}$,
\item $\R^4 \times \R^{4k} \to \R^{4k}$,
\item $\R^8 \times \R^{8k} \to \R^{8k}$,
\item $\R^{p+1} \times \R^{q+1} \to \R^{p+q+1}$,
\item $\R^{p+1} \times \R^{q+1} \to \R^{p+q}$, \ when $p$ and $q$ are both odd,
\item $\R^9 \times \R^{16} \to \R^{16}$.
\end{compactenum}
\end{lemma}

\begin{proof} The nonsingular bilinear maps in items (a)--(d) are induced by real, complex, quaternionic, and octonionic multiplication.  For item (e), treat $\R^{p+1}$ and $\R^{q+1}$ as the coefficients of degree-$p$ and degree-$q$ real polynomials.  Then polynomial multiplication is a nonsingular bilinear map into $\R^{p+q+1}$, the space of degree-$(p+q)$ real polynomials.  When $p$ and $q$ are odd, the similarly defined complex polynomial multiplication may be used for item~(f).  Item (g) follows from the known value $\rho(16) = 9$, which may be computed by the definition of the Hurwitz-Radon function $\rho$ given in Section \ref{sec:nonsingularhistory}.
\end{proof}

Suppose there exists a nonsingular bilinear map $\R^{p+1} \times \R^{q+1} \to \R^{p+q}$ (compare with Lemma \ref{lem:bilinear-exist}(e) and (f)).  Then there exists a biskew map $S^p \times S^q \to S^{p+q-1}$, which induces a map of real projective spaces.  By studying the induced map on cohomology in $\Z/2$-coefficients, Hopf~\cite{Hopf} showed that such a map can only exist if $p$ and $q$ both have a common one in some digit of their binary expansions.  The following more general statement has a nearly identical proof, though we present it in modern language.

The \emph{Stiefel-Whitney height} $\cohomind(X)$ of a $\Z/2$-space $X$ is the largest number $q$ such that the $q^{th}$ power of the first Stiefel-Whitney class is nonzero: $\omega_1(X)^q \neq 0$; here the multiplication is the cup product in the cohomology ring $H^*(X / \Z/2 ; \Z/2)$.  The Stiefel-Whitney height is also (and perhaps more commonly) known as the \emph{cohomological index}, but we use the former term here to avoid confusion with the coindex.  We refer to Chapter 5.3 of \cite{Matousek} for basic properties of the Stiefel-Whitney height.

\begin{proposition}
\label{prop:cohomind}
Let $X_1,\dots,X_k$ be free $\Z/2$-spaces with Stiefel-Whitney heights $i_1,\dots,i_k$.  If no two of the numbers $i_1,\dots, i_k$ share a one in any digit of their binary expansions, then every $(\Z/2)^k$-equivariant map
\[
f \colon X_1 \times \cdots \times X_k \to \R^{i_1+\cdots+i_k}
\]
has a zero.
\end{proposition}

\begin{proof} Let $n = i_1 + \cdots + i_k$.  We show the contrapositive.  If $f$ avoids zero, then by normalization, $f$ maps $(\Z/2)^k$-equivariantly into $S^{n-1}$, hence induces a map on quotients. The quotient map induces a map in cohomology with $\Z/2$-coefficients:
\[
H^*(\R P^{n-1}) \to H^*(X_1/\Z/2 \times \cdots \times X_k/\Z/2) \simeq H^*(X_1/\Z/2) \otimes \cdots \otimes H^*(X_k/\Z/2),
\]
which maps
\[
\omega_1(\R P^{n-1}) \mapsto \omega_1(X_1) \otimes 1 \otimes \cdots \otimes 1  + \cdots + 1 \otimes \cdots \otimes 1 \otimes \omega_1(X_k).
\]
Since $0 = \omega_1(\R P^{n-1})^n$, its image is zero in $H^*(X_1/\Z/2 \times \cdots \times X_k/\Z/2)$:
\[
0 = (\omega_1(X_1) + \cdots + \omega_1(X_k))^n.
\]
Therefore
\[
0 = {{n}\choose{i_1,\dots,i_k}} \omega_1(X_1)^{i_1}\cdots\omega_1(X_k)^{i_k},
\]
which, by the condition on the Stiefel-Whitney heights, implies that the multinomial coefficient ${n}\choose{i_1,\dots,i_k}$ is even.  By a standard generalization of the Lucas theorem \cite{Lucas}, this occurs if and only if two of the $i_1,\dots, i_k$ share a one in any slot of their binary expansions.
\end{proof}

We also require a similar statement which takes into account different possible $\Z/2$-actions in the codomain.  Let $V_{--}, V_{+-}$, or $V_{-+}$ denote $\R$ as a $(\Z/2)^2$-module, where the subscript indicates the action of the two standard generators, that is, $V_{+-}$ indicates that the first generator acts trivially, while the second generator acts by $x \mapsto -x$. The following lemma follows easily from Ramos~\cite{Ramos}. Another short proof using mapping degrees can be found in~\cite{FrickEtal}.

\begin{lemma}
\label{lem:product-of-spheres}
	Let $k,\ell,m,p$, and $q$ be nonnegative integers with $p+q  = k + \ell +m$, $p \ge \ell$, and $q \ge m$. Suppose that $p-\ell$ and $q-m$ do not share a one in any digit of their binary expansions. Then any $(\Z/2)^2$-equivariant map $S^p \times S^q \to V_{--}^k \times V_{-+}^\ell \times V_{+-}^m$ has a zero.
\end{lemma}

\subsection{Joins and deleted joins}
Recall that for topological spaces $X$ and $Y$ the \emph{join} is the space $X * Y$ obtained from $X \times Y \times [0,1]$ by taking the quotient with respect to the equivalence relation generated by $(x,y,0) \sim (x',y,0)$ for $x,x' \in X$ and $(x,y,1) \sim (x,y',1)$ for $y,y' \in Y$. If the simplicial complexes $\Sigma_X$ and $\Sigma_Y$ triangulate $X$ and~$Y$, respectively, then the join $X * Y$ is naturally triangulated by the \emph{join} of simplicial complexes~$\Sigma_X * \Sigma_Y$. As abstract simplicial complexes 
$$\Sigma_X * \Sigma_Y = \{\sigma \times \{1\} \cup \tau \times \{2\} \ : \ \sigma \in \Sigma_X, \ \tau \in \Sigma_Y\},$$
that is, the join is defined by the rule that the vertices of a face in $\Sigma_X$ and the vertices of a face in $\Sigma_Y$ together span a face in~$\Sigma_X * \Sigma_Y$. Here we assume that $\Sigma_X$ and $\Sigma_Y$ have disjoint vertex sets. 

One may think of the join $X * Y$ as abstract convex combinations of points in $X$ and in~$Y$. We will write $\lambda_1x+\lambda_2y$, with $\lambda_1, \lambda_2 \ge 0$ and $\lambda_1 + \lambda_2 = 1$, for the point $(x,y,\lambda_1)$ in~$X * Y$. Thus for $\lambda_1 = 0$ the point $x$ does not influence the point in~$X * Y$, whereas for $\lambda_1 = 1$ the choice of~$y$ does not matter. Note that in this notation $\lambda_1x_1+\lambda_2x_2$ and $\lambda_2x_2+\lambda_1x_1$ determine different points in $X * X$ if $x_1 \ne x_2$.

The \emph{deleted join} $X^{*2}_\Delta$ of a space $X$ is obtained from $X * X$ by deleting all points of the form $(x,x,t)$ with $x \in X$ and $t \in (0,1)$. The \emph{deleted join} $\Sigma^{*2}_\Delta$ of a simplicial complex $\Sigma$ is obtained from the join $\Sigma * \Sigma$ by deleting all those faces that have a vertex in $\Sigma$ in common, that is,
$$\Sigma^{*2}_\Delta = \{\sigma \times \{1\} \cup \tau \times \{2\} \ : \ \sigma, \tau \in \Sigma, \ \sigma \cap \tau = \emptyset\}.$$
In particular the deleted join $(\Delta_{n})^{*2}_\Delta$ of the $n$-simplex $\Delta_n$ is the boundary of an $(n+1)$-dimensional crosspolytope, and thus $(\Delta_{n})^{*2}_\Delta$ is homeomorphic to~$S^n$. Interchanging the join factors $\lambda_1x_1 + \lambda_2x_2 \mapsto \lambda_2x_2 + \lambda_1x_1$ is the antipodal action on the sphere~$(\Delta_{n})^{*2}_\Delta$.

\section{Proof of the main result}
\label{sec:proof}

Here we prove Theorem~\ref{thm:coindex}.  We first state and prove a more general result, Theorem~\ref{thm:index} below, and then derive Theorem~\ref{thm:coindex} as a simple corollary.

\begin{theorem}
\label{thm:index}
	Let $\Sigma$ be a simplicial complex on ground set~$[n]$. Let $c$, $d$, and $\ell$ be nonnegative integers such that $m = d-n+c+2 \le \ell$. Suppose there is a map $\Psi\colon (\Delta_{n-1})^{*2}_\Delta \to \R^c$ with ${\Psi(\lambda_1x_1 + \lambda_2x_2)} = -\Psi(\lambda_2x_2 + \lambda_1x_1)$ such that $\Psi(\lambda_1x_1 + \lambda_2x_2) = 0$ implies $\lambda_1 = \frac12 = \lambda_2$ and ${x_1, x_2 \in \Sigma}$. If the integers $m$ and $\ell-m$ do not share ones in their binary expansions, then $\AEmb_\ell(\Sigma, \R^d)$, and thus $\Emb_\ell(\Sigma, \R^d)$, has coindex at most~${m-1}$.
\end{theorem}

\begin{proof}
	Let $F\colon \Sigma \times S^m \to \R^d$ be a map that is equivariant in the second factor, that is, $F(x,-y) = \varepsilon\cdot F(x,y)$, where $\varepsilon$ acts on $\R^d$ by multiplication by $-1$ in exactly $\ell$ coordinates. We have to show that for some $y\in S^m$ the map $F|_{\Sigma \times \{y\}}$ is not an embedding. 
	
	Extend $F$ continuously to $\Delta_{n-1} \times S^m$ such that the resulting map is antipodal in the second factor. 
	For example, consider the barycentric subdivision of the simplex, and define this extension of $F$ by mapping every barycenter of a face that is not contained in~$\Sigma$ to the origin. Extend this linearly onto the faces of the barycentric subdivision.
	Define $\Phi\colon (\Delta_{n-1})^{*2}_\Delta \times S^m \to \R^{c+1} \times \R^d$ by
	$$\Phi(\lambda_1x_1+\lambda_2x_2, y) = (\Psi(\lambda_1x_1+\lambda_2x_2), \lambda_1F(x_1,y) - \lambda_2F(x_2,y)).$$
    	The domain $(\Delta_{n-1})^{*2}_\Delta \times S^m$ has a free $(\Z/2)^2$-action: One generator $\varepsilon$ acts antipodally on the $S^m$-factor and trivially on $(\Delta_{n-1})^{*2}_\Delta$, while the other generator $\delta$ acts trivially on $S^m$ and by $\delta\cdot (\lambda_1x_1+\lambda_2x_2) = \lambda_2x_2+\lambda_1x_1$ on~$(\Delta_{n-1})^{*2}_\Delta$. In the codomain let $\varepsilon$ act trivially on the first factor of~$\R^{c+1} \times \R^d$, and let $\delta$ acts by multiplication by $-1$ on both factors of~$\R^{c+1} \times \R^d$. Then the map $\Phi$ is $(\Z/2)^2$-equivariant. The generator $\varepsilon$ acts trivially on $c+1+(d-\ell)$ coordinates of $\R^{c+1} \times \R^d$. By Lemma~\ref{lem:product-of-spheres} the map $\Phi$ has a zero, since $m$ and $n-2-c-d+\ell = \ell -m$ do not share any ones in any digit of their binary expansions.
		
	Thus $\Phi$ maps some point $(\lambda_1x_1+\lambda_2x_2, y)$ to zero. Then since $\Psi(\lambda_1x_1+\lambda_2x_2) = 0$ we have that $\lambda_1 = \lambda_2 = \frac12$ and that $x_1,x_2 \in \Sigma$. The last $d$ components of $\Phi$ then imply that $F(\cdot, y)$ cannot be an embedding.
\end{proof}

\begin{lemma}
\label{lem:Kneser-map}
	Let $\Sigma$ be a simplicial complex on ground set~$[n]$, and let $c = \chi(\KG(\Sigma))$. Then there is a map $\Psi\colon (\Delta_{n-1})^{*2}_\Delta \to \R^c$ with $\Psi(\lambda_1x_1 + \lambda_2x_2) = -\Psi(\lambda_2x_2 + \lambda_1x_1)$ such that $\Psi(\lambda_1x_1 + \lambda_2x_2) = 0$ implies $\lambda_1 = \frac12 = \lambda_2$ and $x_1, x_2 \in \Sigma$.
\end{lemma}

\begin{proof}
	Color the missing faces of $\Sigma$ by $\{1, 2, \dots, c\}$ in such a way that disjoint missing faces receive distinct colors. Let $\Sigma_j$ be the simplicial complex on ground set $[n]$ whose missing faces are the missing faces of $\Sigma$ colored~$j$. Then $\Sigma = \Sigma_1 \cap \dots \cap \Sigma_c$.
	
	Define the map $\Psi\colon (\Delta_{n-1})^{*2}_\Delta \to \R^{c+1}$ by $$\Psi(\lambda_1x_1+\lambda_2x_2) = (\lambda_1-\lambda_2, \lambda_1\mathrm{dist}(x_1,\Sigma_1)-\lambda_2\mathrm{dist}(x_2, \Sigma_1), \dots, \lambda_1\mathrm{dist}(x_1,\Sigma_c)-\lambda_2\mathrm{dist}(x_2, \Sigma_c)).$$
	
	Observe that $\Psi(\lambda_1x_1+\lambda_2x_2) = 0$ implies $\lambda_1 = \lambda_2 = \frac12$ and thus $\mathrm{dist}(x_1,\Sigma_j) = \mathrm{dist}(x_2,\Sigma_j)$ for all $j \in [c]$. Since by definition of $(\Delta_{n-1})^{*2}_\Delta$ the points $x_1$ and $x_2$ are in disjoint faces of~$\Delta_{n-1}$ and the missing faces of $\Sigma_j$ intersect pairwise, for every $j$ either $x_1 \in \Sigma_j$ or $x_2 \in \Sigma_j$. This implies $\mathrm{dist}(x_1,\Sigma_j) = \mathrm{dist}(x_2,\Sigma_j) =0$ and thus $x_1, x_2 \in \Sigma$.
\end{proof}

\begin{proof}[Proof of Theorem~\ref{thm:coindex}]
	Combine Theorem~\ref{thm:index} and Lemma~\ref{lem:Kneser-map}.
\end{proof}

\section{Coindex via nonsingular maps}
\label{sec:nonsingular}

Here we collect bounds on $\co(\Emb_d(M,\R^d))$ for topological spaces $M$, using the preliminary results stated in Section \ref{sec:nonsingularprelim}.  Combining these bounds with Theorem \ref{thm:coindex} will yield proofs for several of the corollaries stated in the introduction.

Recall that a nonsingular skew-linear map $\R^{p+1} \times S^q \to \R^d$ induces a $\Z/2$-equivariant map $S^q \to \Emb_d(\R^{p+1},\R^d)$, hence $\co(\Emb_d(\R^{p+1},\R^d)) \geq q$.  More generally:

\begin{lemma}
\label{lem:lowerbound}
Suppose that $M$ embeds in $\R^{p+1}$ and there exists a nonsingular skew-linear map $\R^{p+1} \times S^q \to \R^d$.  Then $\co(\Emb_d(M,\R^d)) \geq q$.  In particular, $\co(\Emb_d(M,\R^d)) \geq d-p-1$.
\end{lemma}

The latter statement follows from Lemma \ref{lem:bilinear-exist}(e).

We also obtain some basic upper bounds on the coindex.  Here
\[
F_2(M) = \{(x,y) \in M \times M \ : \ x \ne y\}
\]
refers to the two-point ordered configuration space of $M$ with the free $\Z/2$-action $(x,y) \mapsto (y,x)$.  Recall that the Stiefel-Whitney height $\cohomind(X)$ was defined in Section \ref{sec:nonsingularprelim}.

\begin{lemma}
\label{lem:obstruct}
Suppose $\cohomind(F_2(M)) \geq p$ and $\co(\Emb_d(M,\R^d)) \geq d-p$.  Then there exists a $(\Z/2)^2$-equivariant map $F_2(M) \times S^{d-p} \to \R^d$ which avoids zero.  In particular, $p$ and $d-p$ share a one in their binary expansions.
\end{lemma}

\begin{proof}
Since $\co(\Emb_d(M,\R^d)) \geq d-p$, there exists a map $f \colon M \times S^{d-p} \to \R^d$ which is $\Z/2$-equivariant in the second factor.  Then $\Phi \colon F_2(M) \times S^{d-p} \to \R^d \colon (x,y,z) \mapsto f(x,z) - f(y,z)$ is $(\Z/2)^2$-equivariant and avoids zero since $f(\cdot,z)$ is an embedding.  By Proposition \ref{prop:cohomind} and the condition on $\cohomind(F_2(M))$, $p$ and $d-p$ share a one in their binary expansions.
\end{proof}

In particular, if $M$ is a topological manifold of dimension $p+1$, there exists a $\Z/2$-equivariant map from $S^p$ to $F_2(M)$, so $p \leq \co(F_2(M)) \leq \cohomind(F_2(M))$.  

\begin{corollary}
\label{cor:embsm}
Suppose the nonnegative integers $p$ and $d-p$ do not share a one in their binary expansions.  If $\cohomind(F_2(M)) \geq p$, then $\co(\Emb_d(M,\R^d)) = d-p-1$.  In particular, $\co(\Emb_d(S^p,\R^d)) = d-p-1$.
\end{corollary}

Recall that Corollary~\ref{cor:parametrized-radon} yields a generalization of the latter statement.

To briefly summarize the above results: the lower bounds for $\co(\Emb_d(M,\R^d))$ are related to the minimal embedding dimension $e$ of $M$ and the existence of skew-linear maps $\R^{e} \times S^q \to \R^d$, whereas the upper bounds for $\co(\Emb_d(M,\R^d))$ are related to the actual dimension $p+1$ of $M$ and the existence of biskew maps $S^p \times S^{d-p} \to \R^d$.  When there is some gap between these dimensions, upper bounds might be improved by computing the Stiefel-Whitney height of the configuration space.

We conclude this section with one result in the smooth category.  By restricting to the space of immersions $\Imm_d(M,\R^d)$ of $C^1$-manifolds~$M$, which contains the space of smooth embeddings $M \to \R^d$, we can also state upper bounds in terms of the existence of skew-linear maps, instead of biskew maps.

\begin{lemma}
\label{lem:smoothemb}
Let $M$ be a $C^1$~manifold of dimension~${p+1}$.  If $\co(\Imm_d(M,\R^d)) \geq q$, then there exists a nonsingular skew-linear map $\R^{p+1} \times S^q \to \R^d$.
\end{lemma}

\begin{proof}
Since $\co(\Imm_d(M,\R^d)) \geq q$, there exists a family of immersions $f_z \colon M \times S^q \to \R^d$ such that $f_z(x) = -f_{-z}(x)$ for all $x \in M$ and $z \in S^q$.  Therefore $(d(f_z))_x = - (d(f_{-z}))_x$, considered as operators $T_xM \to \R^d$, where we have identified the tangent spaces of $\R^d$ with $\R^d$ itself.  Fix $x_0 \in M$.  Define $\Phi \colon T_{x_0}M \times S^q \to \R^d \colon (v,z) \mapsto (d(f_z))_{x_0}(v)$.  This is linear in $v$, $\Z/2$-equivariant in~$z$, and because each $f_z$ is an immersion, only maps to zero when $v=0$.  Thus $\Phi$ is a nonsingular skew-linear map $\R^{p+1} \times S^q \to \R^d$.
\end{proof}

\begin{corollary}
The minimum dimension $q$ such that there exists a nonsingular skew-linear map $\R^{p+1} \times S^q \to \R^d$ is equal to $\co(\Imm_d(\R^{p+1},\R^d))$.
\end{corollary}

\begin{proof} Combine Lemmas~\ref{lem:lowerbound} and~\ref{lem:smoothemb}.
\end{proof}

There exists a nonsingular skew-linear map $\R^{p+1} \times S^q \to \R^d$ if and only if $d\xi_{q}$ has $p+1$ linearly independent sections; here $\xi_q \to \R P^q$ is the canonical line bundle.  Thus determining the coindex of the space $\Imm_d(\R^{p+1},\R^d)$ is closely related to the generalized vector field problem.

It is known that the existence problems for nonsingular bilinear maps and nonsingular skew-linear maps are not equivalent.  In particular, Gitler and Lam~\cite{GitlerLam} showed that there exists a nonsingular skew-linear map $S^{27} \times \R^{13} \to \R^{32}$ (by showing that there exist $13$ linearly independent sections of $32\xi_{27}$) whereas there is no nonsingular bilinear map $\R^{28} \times \R^{13} \to \R^{32}$.

\section{Applications of Theorem \ref{thm:coindex}}
\label{sec:appl}

Here we prove the corollaries stated in Section~\ref{sec:intro}.  The real projective plane~$\R P^2$ can be triangulated in a unique way by a six-vertex triangulation, the antipodal quotient of the icosahedron. We will denote this triangulation by~$\Sigma_{\R P^2}$. This is the smallest triangulation (in the sense of simplicial complexes) of~$\R P^2$. Similarly, $\C P^2$ has a unique minimal triangulation on nine vertices that we will denote by~$\Sigma_{\C P^2}$. The triangulation $\Sigma_{\C P^2}$ was found by K\"uhnel, and the first (computer-aided) proof of its uniqueness is due to K\"uhnel and Lassmann~\cite{Kuhnel1983}; see also~\cite{Kuhnel1983-2}. Arnoux and Marin~\cite{Arnoux1991} showed that any triangulation $\Sigma$ of a $d$-manifold different from the sphere $S^d$ on $3\tfrac{d}{2}+3$ vertices has the property that for any bipartition of the vertex set of $\Sigma$ precisely one part is a face of~$\Sigma$. Thus for such triangulations, in particular for the triangulations $\Sigma_{\R P^2}$ and $\Sigma_{\C P^2}$, no two nonfaces are disjoint, and thus the Kneser graphs $\KG(\Sigma_{\R P^2})$ and $\KG(\Sigma_{\C P^2})$ have no edges. In particular, $\chi(\KG(\Sigma_{\R P^2})) = 1 = \chi(\KG(\Sigma_{\C P^2}))$.

\begin{proof}[Proof of Corollary \ref{cor:coindex-rp2}]
The statement for $\R P^2$ will follow from:
\begin{compactitem}
\item For $d = 4k$, $\co(\Emb_d(\R P^2,\R^d)) \geq 4k-1$,
\item For $d = 4k + 3$, $\co(\AEmb_d(\Sigma_{\R P^2},\R^d)) \leq 4k-1$,
\end{compactitem}
since every embedding is an almost-embedding and the coindex is nondecreasing in $d$.

The first item is a consequence of Lemma \ref{lem:lowerbound} with $p = 3$, together with Lemma \ref{lem:bilinear-exist}(c).  For the second item, we apply Theorem \ref{thm:coindex} to $\Sigma_{\R P^2}$.  In particular, we have $n = 6$, $\chi(\KG(\Sigma_{\R P^2})) = 1$, $\ell = d = 4k + 3$, and $m = d-3 = 4k$.  Since the integers $m = 4k$ and $\ell - m = 3$ do not share any common ones in their binary expansions, the coindex is at most $4k-1$, as desired.

Similarly, the statement for $\C P^2$ follows from:
\begin{compactitem}
\item For $d = 8k$, $\co(\Emb_d(\C P^2,\R^d)) \geq 8k-1$, \qquad (Lemma \ref{lem:lowerbound} ($p=7$) with Lemma \ref{lem:bilinear-exist}(d))
\item For $d = 8k+7$, $\co(\Emb_d(\C P^2,\R^d)) \geq 8k$, \qquad (Lemma \ref{lem:lowerbound} ($p=6$) with Lemma \ref{lem:bilinear-exist}(e))
\end{compactitem}
and the upper bounds from applying Theorem \ref{thm:coindex} to $\Sigma_{\C P^2}$:
\begin{compactitem}
\item For $d = 8k+6$, $\co(\AEmb_d(\Sigma_{\C P^2},\R^d)) \leq 8k-1$, 
\item For $d = 8k+7$, $\co(\AEmb_d(\Sigma_{\C P^2},\R^d)) \leq 8k$.
\end{compactitem}
This completes the proof.
\end{proof}

Observe that the proof of Corollary \ref{cor:coindex-rp2} depends on two elements: the upper bounds arise due to the minimal triangulations of the projective spaces, and the lower bounds arise because the existence of nonsingular bilinear maps is well understood in low dimensions (where these spaces embed).  For higher-dimensional manifolds, it could be nontrivial to obtain strong bounds on both ends.
 
\begin{proof}[Proof of Corollary~\ref{cor:rp2}]
We again apply Theorem~\ref{thm:coindex} to $\Sigma_{\R P^2}$ (resp.\ $\Sigma_{\C P^2}$), this time with $\ell=1$ and $d = 4$ (resp.\ $d= 7$).
\end{proof}

\begin{remark}
	Similar to $\Sigma_{\R P^2}$ and~$\Sigma_{\C P^2}$, there are triangulations of $8$-manifolds different from the sphere on $15$ vertices; see Brehm and K\"uhnel~\cite{Brehm1992}. However, there are now at least six (combinatorially different but PL homeomorphic) such triangulations; see Lutz~\cite{Lutz}. Brehm and K\"uhnel conjectured that these complexes indeed triangulate~$\mathbb{H} P^2$, the quaternionic projective plane, but it is only known that these are triangulations of manifolds like the projective plane in the sense of Eells and Kuiper~\cite{Eells}. In particular, they are cohomology quaternionic planes. Triangulations of $8$-manifolds on $15$ vertices embed into~$\R^{13}$, since they can be realized as proper subcomplexes of~$\partial\Delta_{14}$, which stereographically projects to~$\R^{13}$. Since these triangulations have no two disjoint nonfaces~\cite{Arnoux1991}, Theorem~\ref{thm:coindex} shows that for any triangulation $\Sigma$ of an $8$-manifold with $15$ vertices and for any dimension $d \equiv 12, 13, 14$, or $15 \mod 16$, $\co(\Emb_d(\Sigma, \R^d))\le d -13$. Matching lower bounds follow from combining Lemma~\ref{lem:lowerbound} with Lemma~\ref{lem:bilinear-exist}(e) for $p=12$ and $q=d-13$.
\end{remark}

\begin{proof}[Proof of Corollary~\ref{cor:vKF}]
The first statement is the special case of Theorem~\ref{thm:coindex} with $d = 2k+1$, $n = 2k+3$, $\ell=1$, and $\chi(\KG(\Delta^{(k)}_{2k+2})) = 1$, since no two missing faces are disjoint.  The second statement is the special case of Theorem~\ref{thm:coindex} with $d = 2k+1$, $n = 3k+3$, $\ell=1$, and $\chi(\KG([3]^{*(k+1)})) = k+1$, since a proper coloring of the Kneser graph is induced by coloring the missing edges in the $i^{\text th}$ copy of $[3]$ with color $i$.
\end{proof}

\begin{proof}[Proof of Corollary~\ref{cor:parametrized-radon}]
The lower bound follows from Corollary~\ref{cor:embsm}.  To obtain the upper bound, first apply Theorem \ref{thm:coindex} to $\Delta_{p+1}$; that is, let $\ell = d$, $n = p+2$, and $\chi(\KG(\Delta_{p+1}))=0$, so that $m = d- p$.  By hypothesis, $m$ and $\ell - m = p$ share no ones in their binary expansions, so 
$\co(\AEmb_d(\Delta_{p+1},\R^d)) \leq d-p-1$.
It remains to observe that $\co(\AEmb_d(\partial \Delta_{p+1},\R^d)) \leq \co(\AEmb_d(\Delta_{p+1},\R^d))$, since any almost-embedding $f : \partial \Delta_{p+1} \to \R^d$ extends to an almost-embedding $\tilde{f} : \Delta_{p+1} \to \R^d$, because the behavior of $\tilde{f}$ on the largest face does not affect whether $\tilde{f}$ is an almost-embedding.
\end{proof}

We conclude with the following table, which gives the values of $\co(\AEmb_d(\partial\Delta_{p+1},\R^d))$, for $p \leq 8$.  The evident pattern in each row continues for all values of~$d$. This generalizes classical results for the existence and nonexistence of nonsingular bilinear maps, as tabulated by Berger and Friedland~\cite{BergerFriedland}.

\begin{figure}[h!]
\begin{tabularx}{\textwidth}{ |c|c| *{24}{>{\centering\arraybackslash}X|} }
$_p \backslash ^d$ & 1 & 2 & 3 & 4 & 5 & 6 & 7 & 8 & 9 & 10 & 11 & 12 & 13 & 14 & 15 & 16 & 17 & 18 & 19 & 20 & 21 & 22 & 23 & 24 & 25 \\
\hline
\hline
1 & & 1 & \rc{1} & 3 & \rc{3} & 5 & \rc{5} & 7 & \rc{7} & 9 & \rc{9} & 11 & \rc{11} & 13 & \rc{13} & 15 & \rc{15} & 17 & \rc{17} & 19 & \rc{19} & 21 & \rc{21} & 23 & \rc{23} \\
\hline
2 & & & \rc{0} & 3 & 3 & \rc{3} & \rc{4} & 7 & 7 & \rc{7} & \rc{8} & 11 & 11 & \rc{11} & \rc{12} & 15 & 15 & \rc{15} & \rc{16} & 19 & 19 & \rc{19} & \rc{20} & 23 & 23 \\
\hline
3 & & & & 3 & 3 & 3 & \rc{3} & 7 & 7 & 7 & \rc{7} & 11 & 11 & 11 & \rc{11} & 15 & 15 & 15 & \rc{15} & 19 & 19 & 19 & \rc{19} & 23 & 23 \\
\hline
4 & & & & & \rc{0} & \rc{1} & \rc{2} & 7 & 7 & 7 & 7 & \rc{7} & \rc{8} & \rc{9} & \rc{10} & 15 & 15 & 15 & 15 & \rc{15} & \rc{16} & \rc{17} & \rc{18} & 23 & 23 \\
\hline
5 & & & & & & 1 & \rc{1} & 7 & 7 & 7 & 7 & 7 & \rc{7} & 9 & \rc{9} & 15 & 15 & 15 & 15 & 15 & \rc{15} & 17 & \rc{17} & 23 & 23 \\
\hline
6 & & & & & & & \rc{0} & 7 & 7 & 7 & 7 & 7 & 7 & \rc{7} & \rc{8} & 15 & 15 & 15 & 15 & 15 & 15 & \rc{15} & \rc{16} & 23 & 23 \\
\hline
7 & & & & & & & & 7 & 7 & 7 & 7 & 7 & 7 & 7 & \rc{7} & 15 & 15 & 15 & 15 & 15 & 15 & 15 & \rc{15} & 23 & 23 \\
\hline
8 & & & & & & & & & \rc{0} & \rc{1} & \rc{2} & \rc{3} & \rc{4} & \rc{5} & \rc{6} & 15 & 15 & 15 & 15 & 15 & 15 & 15 & 15 & \rc{15} & \rc{16} \\
\end{tabularx}
\caption{The coindex of $\AEmb_d(\partial\Delta_{p+1},\R^d)$.  The red circled numbers are filled by Corollary~\ref{cor:parametrized-radon}.  The $(8,16)$ entry, along with the rest of the $p=8$ row, may be filled using Lemma~\ref{lem:bilinear-exist}(g) and the fact that the coindex is nondecreasing in the variable~$d$.  Similarly, Lemma~\ref{lem:bilinear-exist}(d) gives a lower bound for the entries in columns $d = 8k$, and together with the nondecreasing fact, the large triangular regions (bounded by column $d = 8k$ and  row $p = 7$) may be filled.  Smaller triangular regions are filled similarly using parts (c), (b), and~(a).  }
\label{fig:embsm}
\end{figure}

\bibliographystyle{plain}

\end{document}